\documentclass[ECP,preprint]{ejpecp} 

\usepackage[T1]{fontenc}
\usepackage[numbers,sort&compress]{natbib}
\usepackage{xcolor}
\usepackage{tikz}

\allowdisplaybreaks

\SHORTTITLE{Can the root cluster remain largest forever?}
\TITLE{Can the root cluster remain largest forever in random recursive tree percolation?\support{The author is supported by China Postdoctoral Science Foundation (No. 2023M743721 and No. 2025T180850).}}
\AUTHORS{Yushu~Zheng\thanks{Academy of Mathematics and Systems Science, Chinese Academy of Sciences. \EMAIL{yszheng666@gmail.com}}}
\AMSSUBJ{60C05; 05C80; 60K35}
\SUBMITTED{}
\ACCEPTED{}
\VOLUME{0}
\YEAR{2023}
\PAPERNUM{0}
\DOI{10.1214/YY-TN}
\ABSTRACT{We consider Bernoulli bond percolation with fixed retention parameter $p$ on the random recursive tree, coupled through the natural growth process.  We prove that the root cluster has a strictly positive probability of remaining a largest cluster at every time.  Equivalently, in the associated Simon-type Chinese restaurant process, the first table has a positive probability of remaining a largest table forever.  We further show that two associated quantities, the limit as $n\to\infty$ of the probability that the root cluster is a largest cluster at time $n$ and the probability that it remains a largest cluster for all times, are strictly increasing and continuous functions of $p$ on $(0,1]$, and both tend to zero as $p\downarrow0$.}

\newcommand{\abs}[1]{\lvert #1 \rvert}
\newcommand{\dd}{\,\mathrm d}
\newcommand{\ee}{\mathrm e}
\newcommand{\eps}{\varepsilon}

\newcommand{\E}{\mathbb E}
\newcommand{\Pp}{\mathbb P}

\newcommand{\1}{\mathbf 1}

\begin{document}

\section{Introduction}
A random recursive tree (RRT) is a basic model of random growth: starting from the root $1$, vertex $n+1$ attaches to a uniformly chosen vertex of $[n]:=\{1,\ldots,n\}$.  We study Bernoulli bond percolation with fixed retention parameter $p\in[0,1]$ on the RRT under its natural growth coupling.  When a new vertex is born, it either joins the open cluster of its parent, or starts a new cluster.  Equivalently, the labelled cluster-size process is a Simon-type Chinese restaurant process: each new customer either sits at the table of a uniformly chosen previous customer, or opens a new table.

Percolation and fragmentation on recursive trees have been studied extensively, mostly from the point of view of cluster sizes.  In the regime $p=p(n)\to1$, Bertoin \cite{bertoin2014sizes,bertoin2014fluctuations} studied the giant root cluster and its fluctuations, while Baur \cite{baur2016percolation} described the limiting cluster genealogy; see also Baur--Bertoin \cite{baur2015fragmentation}.  For fixed $p$, the cluster sizes have power-law behaviour.  Kalay and Ben-Naim \cite{kalay2015fragmentation} predicted this in a related fragmentation model, and Gu and Yuan \cite{gu2024size} recently proved fixed-parameter results for site percolation, including Yule--Simon tails and order $n^p$ growth of the largest clusters.

Here we ask a different, dynamical question: can the root cluster remain largest forever in random recursive tree percolation?  This is not implied by cluster-size asymptotics.  There are infinitely many challenger clusters, and for fixed $p<1$ the root cluster lives on the same $n^p$ scale as the extreme clusters.  Our main result answers the question affirmatively: despite these challengers, the root cluster has a strictly positive probability of remaining a largest cluster forever.

\smallskip
\noindent\textbf{Setup and main results.}
Fix $p\in[0,1]$.  We construct the RRTs $(\mathcal T_n:n\ge1)$ on a common probability space as follows.  Given $\mathcal T_n$, vertex $n+1$ attaches to a uniformly chosen parent in $[n]$.  Independently, the new edge is declared open with probability $p$ and closed with probability $1-p$.  This gives the canonical coupling of bond percolation on the growing RRT.

Clusters are labelled dynamically.  The root cluster has label $1$.  Whenever a new vertex is connected to its parent by an open edge, it inherits the cluster label of its parent; whenever the edge is closed, the new vertex starts the next unused cluster label.  Let
\[
        \Pi^{(n)}=(\Pi^{(n)}_j:j\ge1)
\]
denote the resulting cluster vertex sets at time $n$, with $\Pi^{(n)}_j=\varnothing$ for labels not yet born.  We rank clusters by the total order
\[
        i\prec_n j
        \quad\Longleftrightarrow\quad
        \abs{\Pi^{(n)}_i}>\abs{\Pi^{(n)}_j}
        \quad\text{or}\quad
        \bigl(\abs{\Pi^{(n)}_i}=\abs{\Pi^{(n)}_j}\text{ and }i<j\bigr),
\]
and write $\Pi^{(n),\downarrow}=(\Pi^{(n),\downarrow}_j:j\ge1)$ for the ranked sequence.
Thus $\Pi^{(n),\downarrow}_1$ is the largest cluster at time $n$, with ties resolved by the smaller birth label.

\begin{theorem}\label{thm:persistence}
For Bernoulli bond percolation with fixed $p\in[0,1]$ on the canonically coupled RRT, define the one-time leadership probabilities and their limiting value by
\[
        q_n(p):=\Pp\bigl(1\in\Pi^{(n),\downarrow}_1\bigr),
        \quad n\ge1,\qquad
        \theta(p):=\lim_{n\to\infty}q_n(p),
\]
whose existence is justified in Remark~\ref{rem:limit-pn}.  Also define
\[
        \rho(p):=\Pp\bigl(1\in\Pi^{(n),\downarrow}_1\text{ for all }n\ge1\bigr).
\]
Then, for every $p\in[0,1]$,
\[
        0<\rho(p)\le \theta(p).
\]
\end{theorem}

\begin{theorem}\label{thm:parameter}
The functions $\rho$ and $\theta$ are strictly increasing and continuous on $(0,1]$.  Moreover,
\[
        \lim_{p\downarrow0}\rho(p)=0,
        \qquad
        \lim_{p\downarrow0}\theta(p)=0.
\]
\end{theorem}

\begin{remark}\label{rem:endpoints}
At the endpoints, the model is deterministic: for $p=0$ all clusters are singletons and the root cluster wins every tie, while for $p=1$ all vertices belong to the root cluster.  Hence
\[
        \rho(0)=\theta(0)=\rho(1)=\theta(1)=1.
\]
Together with Theorem~\ref{thm:parameter}, this shows that both functions are discontinuous at $p=0$ and continuous at $p=1$.
\end{remark}

Figure~\ref{fig:parameter-profile} summarizes the qualitative behaviour of the two functions.

\begin{center}
\begin{minipage}{0.92\textwidth}
\centering
\refstepcounter{figure}
\begin{tikzpicture}[x=4.106cm,y=2.404cm]
    \draw[->, line width=0.6pt] (0,0) -- (1.21,0) node[below] {$p$};
    \draw[->, line width=0.6pt] (0,0) -- (0,1.23) node[above] {$\rho,\theta$};
    \draw (0,0) node[below left] {$0$};
    \draw (1,0) -- +(0,0.015) node[below=3pt] {$1$};
    \draw (0,1) -- +(0.015,0) node[left=3pt] {$1$};
    \draw[line width=1.0pt, blue]
        plot[smooth] coordinates {(0,0) (0.20,0.015) (0.40,0.065) (0.60,0.17) (0.78,0.34) (0.92,0.62) (1.00,1.00)};
    \draw[line width=1.0pt, red!75!black]
        plot[smooth] coordinates {(0,0) (0.20,0.035) (0.40,0.13) (0.60,0.30) (0.78,0.55) (0.92,0.83) (1.00,1.00)};
    \node[blue,below right=1pt] at (0.78,0.39) {$\rho(p)$};
    \node[red!75!black,above left=1pt] at (0.84,0.70) {$\theta(p)$};
    \fill (0,1) circle (1.8pt);
    \fill (1,1) circle (1.8pt) node[anchor=south west, xshift=0.5pt, yshift=0.25pt] {$(1,1)$};
    \draw[fill=white,line width=0.7pt] (0,0) circle (1.8pt);
\end{tikzpicture}

\smallskip
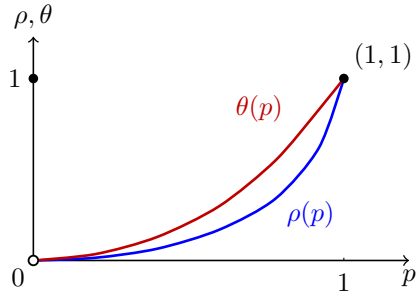
\parbox{\textwidth}{\small\textbf{Figure~\thefigure.} A schematic sketch of the functions $\rho(p)$ and $\theta(p)$.}
\label{fig:parameter-profile}
\end{minipage}
\end{center}

\begin{remark}\label{rem:limit-pn}
We justify why the limit defining $\theta(p)$ exists.  At $p=0$ and $p=1$ this is immediate from Remark~\ref{rem:endpoints}, so fix $0<p<1$.  By Baur--Bertoin \cite[Theorem~2(i) and Section~3.2]{baur2015fragmentation}, with $p=\ee^{-t}$, one has, for every $q>1/p$,
\[
        n^{-p}\bigl(\abs{\Pi_j^{(n)}}:j\ge1\bigr)
        \longrightarrow X=(X_j:j\ge1)
        \qquad\text{a.s. in }\ell^q.
\]
Since $X_1>0$ a.s. and $X\in\ell^q\subset c_0$, the maximum of $X$ is attained; the Mittag--Leffler/Beta coordinate representation of Baur--Bertoin gives $\Pp(X_i=X_j)=0$ for $i\ne j$, hence no tie for the maximum.  The argmax indicator is therefore continuous at $X$, and dominated convergence yields
\[
        q_n(p)\longrightarrow
        \Pp\bigl(X_1>\sup_{j\ge2}X_j\bigr),
        \qquad\text{so}\qquad
        \theta(p)=\Pp\bigl(X_1>\sup_{j\ge2}X_j\bigr).
\]

This identity alone does not show that $\theta(p)>0$, since it only identifies the limiting competition among countably many dependent coordinates.
\end{remark}

\smallskip
\noindent\textbf{Proof outline.}
The event that the root cluster remains largest forever means that it beats every other cluster.  As intuition suggests, these cluster-by-cluster events are positively correlated; see Proposition~\ref{prop:posicor} for a precise statement.  For a single challenger cluster, the aggregation property (Theorem~\ref{thm:aggregation}) shows that its chance of being beaten depends only on the size of the root cluster when the challenger is born.  Finally, Proposition~\ref{prop:io}, proved from the Yule embedding, ensures that the root cluster is typically large enough when later challenger clusters are born, so the resulting failure probabilities are summable.  We carry this out in the equivalent restaurant representation of Proposition~\ref{prop:cluster-crp}, which isolates the labelled cluster-size dynamics and makes the comparisons easier to see.

\smallskip
\noindent\textbf{Organization.}
Section~\ref{sec:prelim} collects the restaurant representation and records some basic properties.  Sections~\ref{sec:persistence-proof} and~\ref{sec:regularity} prove Theorems~\ref{thm:persistence} and~\ref{thm:parameter}, respectively.

\section{Preliminaries}\label{sec:prelim}
This section records the basic tools used throughout the proofs.  We begin with the Simon-type restaurant representation of the labelled cluster process, and then record two structural properties of this representation.

In the restaurant formulation, customers arrive one at a time, and the $i$-th customer sits at table $C_i$.  The seating rule is as follows.  Set $C_1=1$, $K_1=1$, and $N_1^{(1)}=1$.  For $n\ge1$, let
\[
        N_j^{(n)}:=\#\{1\le i\le n:C_i=j\},
        \qquad j\ge1,
\]
the number of customers at table $j$ after $n$ arrivals, so that $K_n=\#\{j:N_j^{(n)}>0\}$ is the number of occupied tables after $n$ customers.  Given
$\mathcal F_n:=\sigma(C_1,\ldots,C_n)$, customer $n+1$ is seated according to
\begin{equation}\label{eq:crp-transition}
\Pp(C_{n+1}=j\mid\mathcal F_n)=
\begin{cases}
        p\,N_j^{(n)}/n, & 1\le j\le K_n,\\[2mm]
        1-p, & j=K_n+1,\\[1mm]
        0, & j>K_n+1.
\end{cases}
\end{equation}
Equivalently, with probability $p$ the new customer chooses one of the existing customers uniformly and sits at the same table; with probability $1-p$ the new customer opens the next table.  We call this the Simon-type Chinese restaurant process with parameter $p$.\footnote{The qualifier ``Simon-type'' distinguishes this process from the classical Ewens--Pitman restaurant: here the probability of opening a new table is the constant $1-p$, while in the Ewens--Pitman process it depends on the current population and, in the two-parameter case, on the number of occupied tables.}

\begin{proposition}\label{prop:cluster-crp}
Let $J_i$ be the birth label of the percolation cluster containing vertex $i$ in the canonically coupled RRT construction above.  Then $(J_i:i\ge1)$ has the same law as the table assignment sequence $(C_i:i\ge1)$.  Consequently, the labelled cluster-size process has the same law as the table-size process:
\[
        \bigl((\abs{\Pi_j^{(n)}}:j\ge1):n\ge1\bigr)
        \overset{d}=
        \bigl((N_j^{(n)}:j\ge1):n\ge1\bigr).
\]
\end{proposition}

\begin{proof}
Conditionally on the RRT percolation history up to time $n$, the parent of vertex $n+1$ is uniform on $[n]$.  Thus the new vertex receives cluster label $j$ with probability $p\abs{\Pi_j^{(n)}}/n$, and receives the next unused label with probability $1-p$.  These are exactly the transitions in \eqref{eq:crp-transition}, so induction on $n$ gives the equality in law of the label processes.
\end{proof}

The second preliminary ingredient is the aggregation, or lumping, property of the restaurant formulation.  It is the sequential version of the classical aggregation property of Dirichlet and Dirichlet--multinomial distributions, together with the corresponding neutrality property for the internal proportions; see Mosimann~\cite{mosimann1962} and Connor--Mosimann~\cite{connor1969}.

Informally, if we group several old tables and regard each group as a single larger table, then the grouped process has the same restaurant dynamics.  Given the grouped process, the internal choices in distinct groups are independent, and each group evolves by size-biased sampling.

\begin{definition}[Aggregated seating sequence]\label{def:aggregated-sequence}
Let $\mathcal B=\{B_r:r\ge1\}$ be a partition of $\mathbb N$ whose blocks are ordered by their least elements:
\[
        \min B_1<\min B_2<\cdots.
\]
Given a seating sequence $C=(C_i:i\ge1)$, define the aggregated seating sequence $C^{\mathcal B}$ by
\[
        C_i^{\mathcal B}=r\quad\text{if and only if}\quad C_i\in B_r.
\]
This is the external, block-level dynamics: it records which block is visited, but not which original table inside that block is chosen.
\end{definition}

\begin{definition}[Within-block seating sequence]\label{def:within-block-sequence}
For a block $B\subset\mathbb N$ and a time $T$, let $m_1<m_2<\cdots$ be the customer labels $m>T$ whose chosen table lies in $B$, and set $C_B^{(T)}=(C_{m_s}:s\ge1)$.
This is the internal dynamics inside $B$ after time $T$: it keeps only the choices among the original tables in $B$.
\end{definition}

We also use the following fixed-table restaurant.  On a finite set of tables $B$, start from deterministic occupancies $n_\ell\ge1$, $\ell\in B$.  Each subsequent customer chooses a table $\ell\in B$ with probability equal to its current occupancy divided by the total occupancy of $B$; no new tables are opened inside $B$.

\begin{theorem}\label{thm:aggregation}
Let $T$ be an a.s. finite stopping time for $(\mathcal F_n)$, and condition on $\mathcal F_T$.  Let $\mathcal B=\{B_r:r\ge1\}$ be a partition such that, for every $k>K_T$, the block containing $k$ is the singleton $\{k\}$ (since the occupied tables at time $T$ have labels $[K_T]$, this amounts to grouping the tables present at time $T$, while leaving all later labels as singletons).  Put
\[
        M_r^{(T)}:=\sum_{\ell\in B_r}N_\ell^{(T)},
        \qquad r\ge1.
\]
Then the following hold.
\begin{enumerate}
\item The aggregated seating sequence after time $T$ is again a Simon-type Chinese restaurant process with parameter $p$, started from the block occupancies $(M_r^{(T)}:r\ge1)$.  Equivalently, for $s\ge T$, an already occupied block $B_r$ is chosen with probability $pM_r^{(s)}/s$, while the next new singleton block is opened with probability $1-p$.
\item Conditionally on the aggregated seating sequence $(C_k^{\mathcal B}:k>T)$, the within-block sequences $C_{B_r}^{(T)}$ over blocks $B_r\subseteq [K_T]$ are independent.  For each such block, $C_{B_r}^{(T)}$ is the fixed-table restaurant on $B_r$ started from the occupancies $(N_\ell^{(T)}:\ell\in B_r)$.
\end{enumerate}
\end{theorem}

\begin{proof}
Given the history up to time $s\ge T$, an already occupied aggregated block $B_r$ is selected with probability
\[
        p\,\frac{\sum_{\ell\in B_r}N_\ell^{(s)}}{s}
        =p\,\frac{M_r^{(s)}}{s},
\]
and a new singleton block is opened with probability $1-p$.  These are exactly the transition probabilities of the aggregated restaurant process.

Now condition also on the aggregated seating sequence.  When the aggregated sequence visits a block $B_r\subseteq[K_T]$ at time $s$, the original table is $\ell\in B_r$ with probability $N_\ell^{(s)}/M_r^{(s)}$.  This transition depends only on the current occupancies inside $B_r$.  Thus the conditional law factorizes over the blocks $B_r\subseteq[K_T]$, and each factor is the law of the corresponding fixed-table restaurant.
\end{proof}

We also need a small correlation fact for the fixed-table restaurants arising after aggregation.  We state it for two tables.  Initially, table $1$ has occupancy $A_0=a\ge1$ and table $2$ has occupancy $B_0=b\ge1$.  At visit $r$, the customer sits at table $1$ with probability $A_{r-1}/(A_{r-1}+B_{r-1})$ and at table $2$ otherwise; the chosen table occupancy is then increased by one.  Let $Y_r$ be the table chosen at visit $r$, and set
\[
        A_r=a+\sum_{s=1}^r\1_{\{Y_s=1\}},
        \qquad
        B_r=b+\sum_{s=1}^r\1_{\{Y_s=2\}}.
\]

\begin{fact}\label{fact:beta-de-finetti}
The sequence $(\1_{\{Y_r=1\}}:r\ge1)$ is exchangeable.  More precisely, there is a random variable $\Theta\sim\mathrm{Beta}(a,b)$ such that, conditionally on $\Theta$, the indicators $\1_{\{Y_r=1\}}$ are i.i.d. Bernoulli$(\Theta)$.
\end{fact}

This is the classical two-colour P\'olya urn representation; see, for example, Pitman~\cite[Section~2.2.2]{pitman2006csp}.
The following estimate is a special case of Burton--Dabrowski~\cite[Theorem~2.1]{burton1992positive}, which shows that infinite exchangeable binary sequences are strong FKG.  For the reader's convenience, we include a short proof.

\begin{proposition}\label{thm:correlation}
For the fixed two-table restaurant above, let
\[
        \xi:=(\1_{\{Y_1=1\}},\1_{\{Y_2=1\}},\ldots)\in\{0,1\}^{\mathbb N}.
\]
Then for any bounded increasing functions $f,g$ on $\{0,1\}^{\mathbb N}$,
\begin{equation}\label{eq:positive-correlation}
        \E\bigl[f(\xi)g(\xi)\bigr]
        \ge
        \E f(\xi)\,\E g(\xi).
\end{equation}
\end{proposition}

\begin{proof}
Condition on the directing variable $\Theta$ from Fact~\ref{fact:beta-de-finetti}.  Given $\Theta=\theta$, the law is a product Bernoulli$(\theta)$ measure, which is positively associated; this follows first for cylinder functions from the finite Harris--FKG inequality and then for bounded increasing functions by monotone-class approximation.  Hence
\[
        \E[f(\xi)g(\xi)\mid\Theta]
        \ge
        \E[f(\xi)\mid\Theta]_{\vphantom{|}}\,\E[g(\xi)\mid\Theta].
\]
For an increasing $f$, the function $\theta\mapsto\E_\theta f$ is increasing by the standard coupling $\1_{\{W_r\le\theta\}}$ with i.i.d. uniform random variables $(W_r)$.  Therefore the two conditional expectations are increasing functions of the same real variable $\Theta$, and Chebyshev's association inequality for monotone functions on the line gives \eqref{eq:positive-correlation}.
\end{proof}

\section{Proof of Theorem 1.1}\label{sec:persistence-proof}
By Proposition~\ref{prop:cluster-crp}, it remains to prove the following statement for the Simon-type restaurant formulation.

\begin{theorem}\label{thm:crp-main}
For the Simon-type Chinese restaurant process \eqref{eq:crp-transition},
\[
        \Pp\bigl(N_1^{(n)}\ge N_j^{(n)}
        \text{ for all }j\ge1\text{ and all }n\ge1\bigr)>0.
\]
Namely, the first table has positive probability to remain a largest table forever.
\end{theorem}

We shall use the following notation.  For $j\ge1$, set
\[
        D_j:=\bigl\{N_1^{(n)}\ge N_j^{(n)}\text{ for all }n\ge1\bigr\},
        \qquad
        T_j:=\inf\{n\ge1:K_n=j\},
        \qquad
        L_j:=N_1^{(T_j)}.
\]
Thus $T_j$ is the arrival time of the first customer at table $j$, and $L_j$ is the size of the first table when table $j$ is opened.

The proof of Theorem~\ref{thm:crp-main} uses the following three estimates.

\begin{proposition}\label{prop:prob}
There exists a constant $c>0$ such that, for every $j\ge2$,
\begin{equation}\label{eq:estimate}
        \Pp(D_j\mid\mathcal F_{T_j})
        \ge 1-\exp(-cL_j)
        \qquad\text{a.s.}
\end{equation}
\end{proposition}

\begin{proposition}\label{prop:posicor}
For $n\ge2$, $m\ge1$, $j_1,\ldots,j_m>n$, and $a_1,\ldots,a_m\ge1$,
\begin{equation}\label{eq:posicor}
\begin{aligned}
&\Pp\Biggl(
        \bigcap_{j=2}^{n}D_j\cap
        \bigcap_{k=1}^{m}\{L_{j_k}\ge a_k\}
        \,\Bigm|\,\mathcal F_{T_n}\Biggr) \\
&\quad\ge
        \Pp\Biggl(
        \bigcap_{j=2}^{n-1}D_j\cap
        \bigcap_{k=1}^{m}\{L_{j_k}\ge a_k\}
        \,\Bigm|\,\mathcal F_{T_n}\Biggr)
        \Pp(D_n\mid\mathcal F_{T_n})
        \qquad\text{a.s.}
\end{aligned}
\end{equation}
\end{proposition}

\begin{proposition}\label{prop:io}
For every $\eps>0$,
\begin{equation}\label{eq:io}
        \Pp\bigl(L_{2^k}<2^{(1-\eps)pk}\text{ for infinitely many }k\bigr)=0.
\end{equation}
\end{proposition}

\begin{proof}[Proof of Theorem~\ref{thm:crp-main} assuming Propositions~\ref{prop:prob}--\ref{prop:io}]
Note that
\begin{equation}\label{eq:representation}
        \bigl\{N_1^{(n)}\ge N_j^{(n)}
        \text{ for all }j\ge1\text{ and all }n\ge1\bigr\}
        =\bigcap_{j=2}^\infty D_j .
\end{equation}
Fix $\eps\in(0,1)$ and put $\alpha=(1-\eps)p$.  By Proposition~\ref{prop:io},
\[
        G_N:=\bigcap_{k=N}^\infty\{L_{2^k}\ge2^{\alpha k}\}
\]
satisfies $\Pp(G_N)>0$ for at least one deterministic $N$.  Fix such an $N$.

For $r\ge N$ and $1\le\ell\le2^r-1$, set
\[
        H_{\ell,r}:=
        \bigcap_{j=2}^{\ell}D_j\cap
        \bigcap_{k=N}^{r}\{L_{2^k}\ge2^{\alpha k}\}.
\]
We claim that there exists $c_1>0$ such that, for every $2\le\ell\le2^r-1$,
\begin{equation}\label{eq:one-step}
        \Pp(H_{\ell,r})
        \ge \bigl(1-\exp(-c_1\ell^\alpha)\bigr)\Pp(H_{\ell-1,r}).
\end{equation}
Indeed, write $q(\ell)=\lfloor\log_2\ell\rfloor$ and put
\[
        A_\ell:=\bigcap_{k=N}^{q(\ell)}\{L_{2^k}\ge2^{\alpha k}\},
\]
with the convention that this is the whole space if $q(\ell)<N$.  By Proposition~\ref{prop:posicor},
\[
\begin{aligned}
\Pp(H_{\ell,r}\mid\mathcal F_{T_\ell})
&=\1_{A_\ell}\,
  \Pp\Biggl(
        \bigcap_{j=2}^{\ell}D_j\cap
        \bigcap_{k=N\vee(q(\ell)+1)}^{r}\{L_{2^k}\ge2^{\alpha k}\}
        \,\Bigm|\,\mathcal F_{T_\ell}\Biggr)\\
&\ge \1_{A_\ell}\,
  \Pp\Biggl(
        \bigcap_{j=2}^{\ell-1}D_j\cap
        \bigcap_{k=N\vee(q(\ell)+1)}^{r}\{L_{2^k}\ge2^{\alpha k}\}
        \,\Bigm|\,\mathcal F_{T_\ell}\Biggr)
  \Pp(D_\ell\mid\mathcal F_{T_\ell}).
\end{aligned}
\]
On $A_\ell$, if $\ell\ge2^N$,
\[
        L_\ell\ge L_{2^{q(\ell)}}\ge 2^{\alpha q(\ell)}
        \ge2^{-\alpha}\ell^\alpha.
\]
For $\ell<2^N$ we use only $L_\ell\ge1$.  Hence, by Proposition~\ref{prop:prob}, there exists $c_1>0$ such that, on \(A_\ell\),
\[
\begin{aligned}
        \Pp(D_\ell\mid\mathcal F_{T_\ell})
        \ge1-\exp(-cL_\ell)\ge1-\exp(-c_1\ell^\alpha),
\end{aligned}
\]
where $c$ is the constant in Proposition~\ref{prop:prob}.  Taking expectations yields \eqref{eq:one-step}.

Iterating \eqref{eq:one-step} for $\ell=2,3,\ldots,2^r-1$ gives
\begin{equation}\label{eq:finite-lower}
        \Pp(H_{2^r-1,r})
        \ge
        \Pp\Bigl(\bigcap_{k=N}^r\{L_{2^k}\ge2^{\alpha k}\}\Bigr)
        \prod_{\ell=2}^{2^r-1}\bigl(1-\exp(-c_1\ell^\alpha)\bigr).
\end{equation}
Letting $r\to\infty$ in \eqref{eq:finite-lower} yields
\[
        \Pp\Bigl(\bigcap_{j=2}^\infty D_j\cap G_N\Bigr)
        \ge
        \Pp(G_N)\prod_{\ell=2}^\infty
        \bigl(1-\exp(-c_1\ell^\alpha)\bigr)>0,
\]
and the conclusion follows from \eqref{eq:representation}.
\end{proof}

We first record a fixed-table estimate.  In the fixed two-table restaurant described above, assume $A_0=n$ and $B_0=1$, and let
\[
        \mathcal D_2=D_2(n):=\{A_r\ge B_r\text{ for all }r\ge0\}.
\]

\begin{lemma}\label{lem:dominate}
There exists a universal constant $c>0$ such that
\[
        \Pp\bigl(\mathcal D_2^c\bigr)\le \exp(-cn),\qquad n\ge1.
\]
\end{lemma}

\begin{proof}
By Fact~\ref{fact:beta-de-finetti}, conditionally on $\Theta$, the difference $A_r-B_r$ is a nearest-neighbour random walk with increments $+1$ with probability $\Theta$ and $-1$ with probability $1-\Theta$, started from $n-1$.  If $\Theta\ge2/3$, the probability that this walk ever makes table $2$ strictly larger than table $1$ is at most $((1-\Theta)/\Theta)^n\le2^{-n}$.  Since $\Theta\sim\mathrm{Beta}(n,1)$,
\[
        \Pp(\Theta<2/3)=(2/3)^n.
\]
For $n=1$ we compute directly
\[
        \Pp(\mathcal D_2^c)
        \le\int_0^{1/2}1\,\dd\theta
        +\int_{1/2}^1\frac{1-\theta}{\theta}\,\dd\theta
        =\log2<1.
\]
For $n\ge2$, the preceding bound gives there exists a universal constant $c>0$ such that
\[
        \Pp\bigl(\mathcal D_2^c\bigr)
        \le2^{-n}+(2/3)^n\le\exp(-cn).\qedhere
\]
\end{proof}

\begin{proof}[Proof of Proposition~\ref{prop:prob}]
At time $T_j$, table $j$ has exactly one customer and table $1$ has $L_j$ customers.  Before time $T_j$, table $j$ is absent, so $D_j$ can fail only after $T_j$.  Apply Theorem~\ref{thm:aggregation} with the partition whose only non-singleton block is $\{1,j\}$.  Conditionally on $\mathcal F_{T_j}$, the future comparison between tables $1$ and $j$, restricted to visits to these two tables, is a fixed two-table restaurant with initial occupancies $L_j$ and $1$.  Lemma~\ref{lem:dominate} gives \eqref{eq:estimate}.
\end{proof}

\begin{proof}[Proof of Proposition~\ref{prop:posicor}]
Let $\mathbb Q$ be a regular conditional law given $\mathcal F_{T_n}$.  Apply Theorem~\ref{thm:aggregation} to the partition
\[
        \mathcal B_n:=\bigl\{\{1,n\},\{2,3,\ldots,n-1\},\{n+1\},\{n+2\},\ldots\bigr\},
\]
with the empty middle block omitted when $n=2$.  Recall the notation $C^{\mathcal B}$ and $C_B^{(T)}$ from Definitions~\ref{def:aggregated-sequence} and \ref{def:within-block-sequence}.  Put
\[
        \mathcal G:=\sigma\bigl((C_k^{\mathcal B_n}:k>T_n),\,C_{\{2,\ldots,n-1\}}^{(T_n)}\bigr),
        \qquad
        \xi:=\bigl(\1_{\{C_{\{1,n\}}^{(T_n)}(r)=1\}}:r\ge1\bigr).
\]
By Theorem~\ref{thm:aggregation}, $\xi$ is independent of $\mathcal G$, and conditionally on $\mathcal G$ its law is that of a fixed two-table restaurant started from $(L_n,1)$.  Let
\[
        E:=\bigcap_{j=2}^{n-1}D_j\cap\bigcap_{k=1}^m\{L_{j_k}\ge a_k\},
        \qquad
        F:=D_n.
\]
Conditionally on $\mathcal G$, there are bounded increasing functions $f,g$ on $\{0,1\}^{\mathbb N}$ such that $\1_E=f(\xi)$ and $\1_F=g(\xi)$.  Hence Proposition~\ref{thm:correlation} gives
\[
        \mathbb Q(E\cap F\mid\mathcal G)
        \ge \mathbb Q(E\mid\mathcal G)\,\mathbb Q(F\mid\mathcal G)
        = \mathbb Q(E\mid\mathcal G)\,\mathbb Q(F).
\]
Taking $\mathbb Q$-expectations yields
\[
        \mathbb Q(E\cap F)\ge \mathbb Q(E)\,\mathbb Q(F),
\]
which is \eqref{eq:posicor}.
\end{proof}

Let $Z^{(q)}$ denote a Yule process with birth rate $q>0$, started from one individual, and let
\[
        \tau_n:=\inf\{t\ge0:Z^{(1)}(t)=n\}.
\]
We shall use the standard Yule embedding from Baur--Bertoin~\cite[First proof of Theorem~2(i)]{baur2015fragmentation}.

\begin{lemma}
There is a coupling of the first-table process with two Yule processes $Z^{(1)}$ and $Z^{(p)}$ such that
\begin{equation}\label{eq:yule-coupling}
        N_1^{(n)}=Z^{(p)}(\tau_n),\qquad n\ge1.
\end{equation}
\end{lemma}

\begin{proof}
For completeness, we recall the short argument.  Run the continuous-time version in which each customer gives births at rate $1$.  Each birth joins the parent's table with probability $p$ and opens a new table with probability $1-p$, independently of everything else.  Observed at birth times, this has the transition probabilities \eqref{eq:crp-transition}.  The total population is a Yule process $Z^{(1)}$, while births that keep the first-table label occur from first-table customers at total rate $p$ times the current first-table size.  Hence the first table evolves as a Yule process $Z^{(p)}$, observed at the stopping times $\tau_n$.
\end{proof}

We shall use the elementary one-dimensional distribution of the Yule process; see, for example, Durrett~\cite[Theorem~4.2, equation~(4.14)]{durrett2016essentials}:
\begin{equation}\label{eq:yule-geometric}
        \Pp(Z^{(q)}(t)=m)=\ee^{-qt}\bigl(1-\ee^{-qt}\bigr)^{m-1},
        \qquad m\ge1.
\end{equation}

\begin{lemma}\label{lem:yule-ld}
For every $\eps\in(0,1)$ and $q>0$, there are constants $C,c>0$ such that, for all $t\ge1$,
\[
        \Pp\bigl(Z^{(q)}(t)<\ee^{(1-\eps)qt}\bigr)\le C\ee^{-c\eps q t},
        \qquad
        \Pp\bigl(Z^{(q)}(t)>\ee^{(1+\eps)qt}\bigr)\le C\ee^{-c\ee^{\eps qt}}.
\]
Moreover, for $n\ge2$,
\[
        \Pp\bigl(\tau_n\notin ((1-\eps)\log n,(1+\eps)\log n)\bigr)
        \le Cn^{-c\eps}.
\]
\end{lemma}

\begin{proof}
The first two bounds follow directly from \eqref{eq:yule-geometric}.  The bounds for $\tau_n$ follow by applying these estimates to $Z^{(1)}((1-\eps)\log n)$ and $Z^{(1)}((1+\eps)\log n)$.
\end{proof}

\begin{lemma}\label{lem:root-ld}
For every $\eps\in(0,1)$, there are constants $C,c>0$ such that, for all $n\ge2$,
\[
        \Pp\bigl(N_1^{(n)}<n^{(1-\eps)p}\bigr)\le Cn^{-c}.
\]
\end{lemma}

\begin{proof}
Choose $\delta\in(0,\eps)$.  By the coupling \eqref{eq:yule-coupling},
\[
        \Pp\bigl(N_1^{(n)}<n^{(1-\eps)p}\bigr)
        \le \Pp\bigl(\tau_n<(1-\delta)\log n\bigr)
        +\Pp\bigl(Z^{(p)}((1-\delta)\log n)<n^{(1-\eps)p}\bigr).
\]
Both terms are $O(n^{-c})$ by Lemma~\ref{lem:yule-ld}, since $(1-\delta)p>(1-\eps)p$.
\end{proof}

\begin{lemma}\label{lem:L-lower}
For every $\eps\in(0,1)$, there are constants $C,c>0$ such that, for all $j\ge2$,
\[
        \Pp\bigl(L_j<j^{(1-\eps)p}\bigr)\le Cj^{-c}.
\]
\end{lemma}

\begin{proof}
After the first customer, the events that a new table is opened at subsequent arrivals are i.i.d. Bernoulli$(1-p)$.  Put $m_j:=\left\lfloor j/(2(1-p))\right\rfloor$.
A Chernoff bound gives
\[
        \Pp(T_j<m_j)\le \exp(-c_1j)
\]
for some $c_1=c_1(p)>0$.  On the event $\{T_j\ge m_j\}$, monotonicity of the first table gives $L_j=N_1^{(T_j)}\ge N_1^{(m_j)}$.  Choose $\delta\in(0,\eps)$.  For all sufficiently large $j$,
\(
        m_j^{(1-\delta)p}\ge j^{(1-\eps)p}.
\)
Therefore, by Lemma~\ref{lem:root-ld}, there exist $C,c>0$ such that
\[
        \Pp\bigl(L_j<j^{(1-\eps)p}\bigr)
        \le \Pp(T_j<m_j)+\Pp\bigl(N_1^{(m_j)}<m_j^{(1-\delta)p}\bigr)
        \le Cj^{-c}.\qedhere
\]
\end{proof}

\begin{proof}[Proof of Proposition~\ref{prop:io}]
By Lemma~\ref{lem:L-lower},
\[
        \sum_{k=1}^\infty
        \Pp\bigl(L_{2^k}<2^{(1-\eps)pk}\bigr)
        \le C\sum_{k=1}^\infty 2^{-ck}<\infty.
\]
The Borel--Cantelli lemma proves \eqref{eq:io}.
\end{proof}

\section{Proof of Theorem 1.2}\label{sec:regularity}
By Proposition~\ref{prop:cluster-crp}, we may couple the RRT percolation model and the Simon-type restaurant so that their labelled size processes agree.  Hence, for $p\in[0,1]$,
\[
        \theta(p)=\lim_{n\to\infty}\Pp\bigl(N_1^{(n)}\ge N_j^{(n)}\text{ for all }j\ge1\bigr),
        \qquad
        \rho(p)=\Pp\bigl(D_j\text{ holds for every }j\ge2\bigr),
\]
where the probabilities on the right-hand side are for the restaurant with parameter $p$.  It is therefore enough to prove the corresponding results in the restaurant model.
Theorem~\ref{thm:parameter} follows directly from Propositions~\ref{prop:regularity-monotonicity}--\ref{prop:regularity-zero-limit} below.

\subsection{Preliminaries}
\noindent\emph{Coupling.}
We introduce a continuous-time version of the restaurant process to couple CRPs with different parameters.  For $p\in(0,1]$, put $\lambda=(1-p)/p$.  In the continuous-time restaurant with parameter $\lambda$, each customer gives births at rate $1$ to its own table and at rate $\lambda$ to a new table.  Observed at birth times, this gives the Simon-type restaurant with parameter $p=1/(1+\lambda)$.  We refer to this continuous-time restaurant as the $\lambda$-process below.

For $0<p_1<p_2\le1$, set
\[
        \lambda_i=\frac{1-p_i}{p_i},
        \qquad
        \delta:=\lambda_1-\lambda_2>0.
\]
We couple the two continuous-time restaurants by starting from the $\lambda_2$-process and adding independent new-table births of rate $\delta$ per customer.  Each extra table is then equipped recursively with its own same-table and new-table birth clocks, so the retained and extra tables together have the law of the $\lambda_1$-process.  Deleting the extra tables and all their descendants recovers the $\lambda_2$-process, while the first table is unchanged.

\smallskip
\noindent\emph{Uniform tail estimate.}
The next lemma controls, uniformly for $p$ bounded away from $0$, the chance
that some sufficiently late table ever overtakes the first table; this is the
tail estimate needed for the finite-table approximation in the continuity
proof.

\begin{lemma}
For every $a\in(0,1]$,
\begin{equation}
        \sup_{p\in[a,1]}\Pp\bigl(\exists j\ge2^M:D_j^c\bigr)\longrightarrow0,
        \qquad M\to\infty.
        \label{eq:regularity-tail-unconditioned}
\end{equation}
\end{lemma}

\begin{proof}
The case $a=1$ is trivial, so assume $a<1$ and choose $\alpha=a/2$.  By the monotone coupling in $p$ and Lemma~\ref{lem:L-lower} applied at $p=a$, there are constants $C,\eta>0$ such that
\begin{equation}\label{eq:regularity-uniform-L-lower}
        \sup_{p\in[a,1)}\Pp\bigl(L_{2^k}<2^{\alpha k}\bigr)\le C2^{-\eta k}.
\end{equation}
With $I_k:=\{2^k,\ldots,2^{k+1}-1\}$, Proposition~\ref{prop:prob} and
\eqref{eq:regularity-uniform-L-lower} give
\[
\begin{aligned}
        \sup_{p\in[a,1]}\Pp\bigl(\exists j\in I_k:D_j^c\bigr)
        &\le \sup_{p\in[a,1)}
        \biggl\{\Pp\bigl(L_{2^k}<2^{\alpha k}\bigr)
        +\sum_{j\in I_k}
        \E_p\bigl[
        \1_{\{L_{2^k}\ge2^{\alpha k}\}}
        \Pp(D_j^c\mid\mathcal F_{T_j})
        \bigr]\biggr\}  \\
        &\le C2^{-\eta k}+2^k\exp(-c2^{\alpha k}),
\end{aligned}
\]
where $\{L_{2^k}\ge2^{\alpha k}\}\in\mathcal F_{T_j}$ and, on this event,
$L_j\ge L_{2^k}\ge2^{\alpha k}$ for $j\in I_k$.  Summing over $k\ge M$ proves the claim.
\end{proof}

\subsection{Proof of Theorem 1.2}
\begin{proposition}\label{prop:regularity-monotonicity}
The functions $\theta$ and $\rho$ are strictly increasing on $(0,1]$.
\end{proposition}

\begin{proof}
\emph{Monotonicity.}  Fix $0<p_1<p_2\le1$ and use the coupling above.  Let $A_i$ be the event that the first table is never overtaken, and let $B_i$ be the event that the first table has the largest limiting weight in the $\lambda_i$-process.  Then
\[
        A_1\subseteq A_2,
        \qquad
        B_1\subseteq B_2.
\]
By definition, $\Pp(A_i)=\rho(p_i)$, while Remark~\ref{rem:limit-pn} gives $\Pp(B_i)=\theta(p_i)$.
Thus $\rho(p_1)\le\rho(p_2)$ and $\theta(p_1)\le\theta(p_2)$.

\emph{Strictness.}  Condition on the $\lambda_2$-process, and let $R(t)$ be the size of its first table.  Then
\[
        W_1:=\lim_{t\to\infty}\ee^{-t}R(t)
\]
exists and belongs to $(0,\infty)$ a.s. by the standard martingale convergence
theorem for Yule processes; see, for example,
Durrett~\cite[Theorem~4.2]{durrett2016essentials}.  In this $\lambda$-parametrization, each table grows after its birth as a rate-one Yule process.  Extra tables born from first-table members during $[0,1]$ form a Poisson process with intensity $\delta R(s)\,\dd s$.  If such a table is born at time $s$, its limiting weight relative to the common factor $\ee^t$ is $\ee^{-s}E_s$, with $E_s\sim{\rm Exp}(1)$ independently.  Hence the event
\[
        H:=\{\exists s\in[0,1]\text{ an extra table with }\ee^{-s}E_s>W_1\}
\]
has conditional probability
\[
        q:=\Pp(H\mid\lambda_2\text{-process})
        =
        1-\exp\biggl\{
        -\delta\int_0^1 R(s)\exp(-\ee^sW_1)\,\dd s
        \biggr\}>0
        \quad\text{a.s.}
\]
Since $\Pp(A_2)=\rho(p_2)>0$,
\[
        \Pp(A_2\cap H)
        =\E\bigl[q\1_{A_2}\bigr]>0.
\]
On $A_2\cap H$, persistence holds for $\lambda_2$ but fails for $\lambda_1$.  Hence
\[
        \rho(p_2)-\rho(p_1)
        \ge \Pp(A_2\cap H)>0.
\]
Similarly, since $\Pp(B_2)=\theta(p_2)>0$,
\[
        \theta(p_2)-\theta(p_1)
        \ge \Pp(B_2\cap H)
        =\E\bigl[q\1_{B_2}\bigr]>0.\qedhere
\]
\end{proof}

To prove continuity, we first record a finite-table approximation.

\begin{lemma}\label{lem:finite-approximants-continuous}
For each fixed $m\ge2$, define, for $0<p<1$,
\[
        \rho_m(p):=\Pp\Bigl(\bigcap_{j=2}^m D_j\Bigr),
        \qquad
        \theta_m(p):=\Pp\bigl(X_1>\max_{2\le j\le m}X_j\bigr),
\]
where $X=(X_j:j\ge1)$ is the labelled scaling limit from
Remark~\ref{rem:limit-pn}.  Set $\rho_m(1)=\theta_m(1)=1$.  Then
$\rho_m$ and $\theta_m$ are continuous on $(0,1]$.
\end{lemma}

\begin{proof}
First fix a compact interval $K=[a,b]\subset(0,1)$.  The random time $T_m$
has an exponentially decaying tail uniformly for $p\in K$: the successive
new-table openings occur with probability $1-p\ge1-b$, so $T_m$ is dominated
by the sum of $m-1$ geometric random variables with parameter $1-b$.

On each event $\{T_m=r\}$ there are only finitely many seating histories
$h=(c_1,\ldots,c_r)$.  Each such history has probability polynomial in $p$,
because every step contributes either a factor $1-p$ for a new table or a
factor $p$ times a deterministic occupancy ratio for an old table choice.
Given $h$, discard all subsequent customers seated outside the first $m$
tables.  By Theorem~\ref{thm:aggregation}, the relative seating process among
tables $1,\ldots,m$ is a fixed $m$-table restaurant whose law depends only on
the occupancies at time $T_m$, not on $p$.  Both events $\cap_{j=2}^mD_j$ and
$\{X_1>\max_{2\le j\le m}X_j\}$ are determined by this relative process.
Thus, conditional on $h$, their probabilities are constants independent of
$p$.

If we truncate to histories with $T_m\le R$, the resulting approximants are
finite sums of polynomials in $p$, hence continuous.  The error is at most
$\Pp(T_m>R)$ for both $\rho_m$ and $\theta_m$, and this tends to zero
uniformly on $K$.  Therefore $\rho_m$ and $\theta_m$ are continuous on
$(0,1)$.

It remains to check continuity at $p=1$.  Fix $R\ge1$.  If the first $R$
customers all sit at table $1$, then $L_j\ge R$ for every fixed $j\ge2$, and
this event has probability $p^{R-1}$.  By Proposition~\ref{prop:prob},
\[
        \Pp(D_j^c)\le 1-p^{R-1}+\ee^{-cR},
        \qquad j\ge2.
\]
Hence
\[
        1-\rho_m(p)
        \le \sum_{j=2}^m\Pp(D_j^c)
        \le (m-1)\bigl(1-p^{R-1}+\ee^{-cR}\bigr).
\]
Letting first $p\uparrow1$ and then $R\to\infty$ gives $\rho_m(p)\to1$.
The finite analogue of $\rho\le\theta$ gives $\rho_m(p)\le\theta_m(p)\le1$,
so also $\theta_m(p)\to1$.
\end{proof}

\begin{proposition}
The functions $\theta$ and $\rho$ are continuous on $(0,1]$.
\end{proposition}

\begin{proof}
Fix $a\in(0,1]$ and put $K=[a,1]$.  Let $\rho_m,\theta_m$ be the finite
approximants from Lemma~\ref{lem:finite-approximants-continuous}.  By the
lemma, both are continuous on $K$.

Note that, for $p\in K$,
\[
        0\le \rho_m(p)-\rho(p)
        \le \Pp\bigl(\exists j>m:D_j^c\bigr),
        \qquad
        0\le \theta_m(p)-\theta(p)
        \le \Pp\bigl(\exists j>m:D_j^c\bigr),
\]
since a later limiting winner must eventually overtake the first table.  By~\eqref{eq:regularity-tail-unconditioned},
\[
        \sup_{p\in K}\Pp\bigl(\exists j>m:D_j^c\bigr)\longrightarrow0.
\]
Thus $\rho_m\to\rho$ and $\theta_m\to\theta$ uniformly on $K$, which proves the
claim since $a$ is arbitrary.
\end{proof}

\begin{proposition}\label{prop:regularity-zero-limit}
As $p\downarrow0$,
\[
        \rho(p)\to0,
        \qquad
        \theta(p)\to0.
\]
\end{proposition}

\begin{proof}
Let $B_M$ be the event that the first $M$ customers occupy $M$ distinct tables.  Then $\Pp(B_M)=(1-p)^{M-1}$.  On $B_M$, the first $M$ tables are exchangeable thereafter, and limiting ties among them have probability zero; hence
\[
        \theta(p)
        \le 1-(1-p)^{M-1}+\frac{(1-p)^{M-1}}{M}.
\]
Letting first $p\downarrow0$ and then $M\to\infty$ gives $\theta(p)\to0$.
Since $0\le\rho(p)\le\theta(p)$, also $\rho(p)\to0$.
\end{proof}

\begin{acks}
The author thanks Chenlin Gu and Xingjian Hu for very helpful discussions.
\end{acks}

\bibliographystyle{amsplain}
\bibliography{mybib}

\end{document}